\documentclass[preprint,sort&compress,3p]{elsarticle}

\usepackage{amsmath}
\usepackage{amsfonts}
\usepackage{amsopn}
\usepackage{amsthm}
\usepackage{textcomp}
\usepackage{color}

\newtheorem{thm}{Theorem}
\newtheorem{lemma}{Lemma}
\newtheorem{cor}{Corollary}
\newtheorem*{remark}{Remark}

\title{Associating the Invariant Subspaces of a Non-Normal Matrix with Transient Effects in its Matrix Exponential or Matrix Powers}

\author{Matthew~G.~Reuter\corref{cor1}}
\ead{matthew.reuter@stonybrook.edu}
\address{Department of Applied Mathematics and Statistics \& Institute for Advanced Computational Science, Stony Brook University, Stony Brook, New York 11794, United States}

\cortext[cor1]{Corresponding author (matthew.reuter@stonybrook.edu)}

\begin{document}

\begin{abstract}
It is well known that the matrix exponential of a non-normal matrix can exhibit transient growth even when all eigenvalues of the matrix have negative real part, and similarly for the powers of the matrix when all eigenvalues have magnitude less than 1. Established conditions for the existence of these transient effects depend on properties of the entire matrix, such as the Kreiss constant, and can be laborious to use in practice. In this work we develop a relationship between the invariant subspaces of the matrix and the existence of transient effects in the matrix exponential or matrix powers. Analytical results are obtained for two-dimensional invariant subspaces and Jordan subspaces, with the former causing transient effects when the angle between the subspace's constituent eigenvectors is sufficiently small. In addition to providing a finer-grained understanding of transient effects in terms of specific invariant subspaces, this analysis also enables geometric interpretations for the transient effects.
\end{abstract}

\begin{keyword}
Matrix exponential \sep matrix powers \sep transient behavior \sep non-normal matrices \sep numerical range \\
AMS Codes: 15A16 \sep 
15A60 
\end{keyword}

\maketitle

\section{Introduction}
\label{sec:intro}
The matrix exponential $e^{\mathbf{M}t}$, for $\mathbf{M}\in\mathbb{C}^{N\times N}$ and $t\in\mathbb{R}$ (often $t\ge0$), appears in many contexts \cite{moler-3-2003, bk:trefethen-2005}, notably when solving a system of first-order, linear ordinary differential equations. The matrix powers $\mathbf{M}^n$ ($n\ge0$) also have wide applicability \cite{bk:trefethen-2005}; for instance in Markov processes or in condensed matter physics \cite{reuter-053001-2017}. When $\mathbf{M}$ is normal (\textit{i.e.},\ $\mathbf{M}$ has a complete set of orthogonal eigenvectors), its eigenvalues completely determine the behavior of $e^{\mathbf{M}t}$ and of $\mathbf{M}^n$ \cite{bk:trefethen-2005}. That is, if all eigenvalues of $\mathbf{M}$ have negative real part, then $\|e^{\mathbf{M}t}\|$ decays monotonically with $t$; likewise, $\| \mathbf{M}^n \|$ decays with $n$ when all eigenvalues of $\mathbf{M}$ have magnitude less than 1.

The situation is more complicated when $\mathbf{M}$ is not normal \cite{bk:trefethen-2005}. The eigenvalues of $\mathbf{M}$ correctly predict asymptotic behaviors as $t\to\infty$ or $n\to\infty$, but can qualitatively fail for small and intermediate values of $t$ or $n$. For example, there can be transient growth in $e^{\mathbf{M}t}$ such that $\|e^{\mathbf{M}t}\|>1$ for some $0 < t \ll \infty$ even though all eigenvalues of $\mathbf{M}$ have negative real part. Similar effects can be observed in $\|\mathbf{M}^n\|$. Numerous applications of these transient effects --- \textit{i.e.},\ $\|e^{\mathbf{M}t}\|>1$ or $\| \mathbf{M}^n \|>1$ --- are discussed in \cite{bk:trefethen-2005}. Note that this definition of transience, although taken from \cite{bk:trefethen-2005}, is different from ``transience'' in some applications, where it may mean that the system has not yet reached asymptotic or steady-state conditions. Throughout this work, transience will mean that there is transient growth in the matrix exponential or matrix powers when none would be predicted from the matrix's eigenvalues.

Pseudospectral analyses \cite{bk:trefethen-2005} help predict the existence of transient growth, but have two limitations. First, they rely on quantities that can be laborious to obtain in practice, such as the condition number of the eigenvector matrix of $\mathbf{M}$ or the Kreiss constant of $\mathbf{M}$ (see chapter 14 of \cite{bk:trefethen-2005}). Second, they do not usually provide sharp conditions for when $\mathbf{M}$ will or will not display transient effects. Furthermore, better predictions of transient effects typically require the more computationally intensive properties of $\mathbf{M}$, \textit{e.g.},\ the Kreiss constant.

The goal of this work is to determine more practical conditions for the existence of transient effects in the matrix exponential or matrix powers of $\mathbf{M}$. Our approach differs significantly from previous studies. Rather than relating ``holistic'' properties of $\mathbf{M}$ (such as the Kreiss constant) to transient effects, we examine the $k$-dimensional ($k\ge 2$) invariant subspaces of $\mathbf{M}$ \cite{bk:gohberg-2006} and show that transient effects can be associated with invariant subspaces that satisfy certain conditions. This ultimately provides a more fine-grained explanation of transient effects and potentially introduces a new line of inquiry for characterizing non-normal matrices in terms of their invariant subspaces (see Thm.\ \ref{thm:subspaces}).

The main results are stated in Thms.\ \ref{thm:transient-exp} and \ref{thm:transient-power} for $e^{\mathbf{M}t}$ and $\mathbf{M}^n$, respectively, and are proven in section \ref{sec:higherD}. Most of these conditions come from analytical results for the 2-dimensional invariant subspaces of $\mathbf{M}$. The layout of this paper is as follows. First, section \ref{sec:prelims} introduces our notation and reviews key concepts from matrix analysis. Section \ref{sec:2by2} then derives conditions for transient effects when $\mathbf{M}\in\mathbb{C}^{2\times2}$. These results are generalized to $\mathbf{M}\in\mathbb{C}^{N\times N}$ in section \ref{sec:higherD}, where they apply to the 2-dimensional invariant subspaces of $\mathbf{M}$. Results for higher-dimensional Jordan subspaces are also presented. Additionally, there are several geometric interpretations for these results, which are discussed in section \ref{sec:geom}. We finally offer concluding remarks and ideas for future studies in section \ref{sec:conclusions}.

\begin{thm}[Transient Effects in the Matrix Exponential]
\label{thm:transient-exp}
Let $\mathbf{M}$ be a $N\times N$ complex matrix ($N\ge 2$) with all of its eigenvalues in the open left half of $\mathbb{C}$; that is, $\mathrm{Re}(\lambda)<0$ if $\lambda$ is an eigenvalue of $\mathbf{M}$. Let $S$ be an invariant subspace of $\mathbf{M}$. Then $e^{\mathbf{M}t}$ will have transient effects when at least one of the following conditions is satisfied.
\begin{subequations}
\begin{enumerate}
\item $S$ is a $M$-dimensional ($M\ge2$) Jordan subspace associated with defective eigenvalue $\lambda$ and
\begin{equation}
\mathrm{Re}(\lambda) > -\cos\left( \frac{\pi}{M+1} \right).
\label{eq:defective-lambda-exp}
\end{equation}

\item $S$ is a 2-dimensional subspace spanned by eigenvectors associated with eigenvalues $\lambda_1$ and $\lambda_2$, $\lambda_1\neq\lambda_2$, $\theta\in(0,\pi/2]$ is the angle between the eigenvectors [in the Hermitian sense of Eq.\ \eqref{eq:vector-angle}], and
\begin{equation}
\theta < \arctan\left( \frac{\left| \lambda_1 - \lambda_2 \right|}{2 \sqrt{\mathrm{Re}(\lambda_1) \mathrm{Re}(\lambda_2)}} \right).
\label{eq:theta-exp}
\end{equation}
\end{enumerate}
\end{subequations}
\end{thm}

\begin{thm}[Transient Effects in the Matrix Powers]
\label{thm:transient-power}
Let $\mathbf{M}$ be a nonsingular $N\times N$ complex matrix ($N\ge 2$) with all of its eigenvalues in the open unit circle of $\mathbb{C}$; that is, $0<|\lambda|<1$ if $\lambda$ is an eigenvalue of $\mathbf{M}$. Let $S$ be an invariant subspace of $\mathbf{M}$. Then $\mathbf{M}^n$ will have transient effects when at least one of the following conditions is satisfied.
\begin{enumerate}
\item $S$ is a $M$-dimensional ($M\ge2$) Jordan subspace.

\item $S$ is a 2-dimensional subspace spanned by eigenvectors associated with eigenvalues $\lambda_1$ and $\lambda_2$, $\lambda_1\neq\lambda_2$, $\theta\in(0,\pi/2]$ is the angle between the eigenvectors [in the Hermitian sense of Eq.\ \eqref{eq:vector-angle}], and
\begin{equation}
\theta < \arctan\left( \frac{\left| \ln (\lambda_1) - \ln(\lambda_2) \right|}{2 \sqrt{ \ln |\lambda_1| \ln |\lambda_2| } } \right).
\label{eq:theta-power}
\end{equation}
The branch of $\ln(z)$ can always be chosen to make the imaginary part of the numerator of Eq.\ \eqref{eq:theta-power} be in $[-\pi,\pi)$.
\end{enumerate}
\end{thm}

\section{Preliminaries and Notation}
\label{sec:prelims}
Several standard results from matrix analysis will be needed for the discussion and proofs of Thms.\ \ref{thm:transient-exp} and \ref{thm:transient-power}. We overview them in this section, which also introduces our notation. Throughout this section, we assume $\mathbf{M}\in\mathbb{C}^{N\times N}$. The angle between two vectors will be calculated in the Hermitian sense \cite{galantai-589-2006} with
\begin{equation}
\cos(\theta) = \frac{\left|v_1 \cdot{} v_2\right|}{\left\| v_1 \right\| \left\| v_2 \right\|}
\label{eq:vector-angle}
\end{equation}
such that $\theta\in[0,\pi/2]$.

\subsection{Schur Decomposition, Unitary Invariance, and Matrix Norm}
\label{sec:prelims:schur}
Every $\mathbf{M}$ is unitarily similar to an upper triangular matrix; that is, $\mathbf{M} = \mathbf{QTQ}^{-1}$, where $\mathbf{Q}$ is unitary and $\mathbf{T}$ is upper triangular. $\mathbf{T}$ is called a Schur form of $\mathbf{M}$ \cite{bk:golub-2013}. It is trivial to show that $\mathbf{M}^n=\mathbf{QT}^n\mathbf{Q}^{-1}$ and, by use of the matrix exponential's Taylor series, $e^{\mathbf{M}t}=\mathbf{Q}e^{\mathbf{T}t}\mathbf{Q}^{-1}$ \cite{bk:higham-2008}. Then, assuming the 2-norm or the Frobenius norm for the matrix norm, we get that $\|e^{\mathbf{M}t}\| = \|e^{\mathbf{T}t}\|$ and $\|\mathbf{M}^n\|=\|\mathbf{T}^n\|$ because these matrix norms are invariant to unitary transformations \cite{bk:golub-2013}. It is therefore sufficient to consider a Schur form of $\mathbf{M}$ in our analyses.

\subsection{Invariant Subspaces and Restrictions}
\label{sec:prelims:subspaces}
A subspace $S\subseteq\mathbb{C}^N$ is an invariant subspace of $\mathbf{M}$ if $\mathbf{M}x\in S$ for every $x\in S$ \cite{bk:gohberg-2006}. The invariant subspaces of $\mathbf{M}$ are closely tied to the Jordan decomposition of $\mathbf{M}$; every invariant subspace of $\mathbf{M}$ is spanned by a set of eigenvectors and generalized eigenvectors of $\mathbf{M}$. In this way, each invariant subspace of $\mathbf{M}$ can be associated with eigenvalues of $\mathbf{M}$. An invariant subspace of $\mathbf{M}$ spanned by eigenvectors and generalized eigenvectors that all belong to the same Jordan chain is called a Jordan subspace of $\mathbf{M}$. A 2- or higher-dimensional Jordan subspace of $\mathbf{M}$ is thus associated with a sole defective eigenvalue of $\mathbf{M}$.

We will also need the restriction of $\mathbf{M}$ onto an invariant subspace $S$, denoted $\left.\mathbf{M}\right|_S:S\to S$ and defined by
\[
\left.\mathbf{M}\right|_S x = \mathbf{M}x
\]
for $x\in S$. Thus, $\left.\mathbf{M}\right|_S$ can be regarded as a $\mathrm{dim}(S)\times\mathrm{dim}(S)$ matrix. If $f$ is a well-behaved function, such as the exponential, then the Jordan form of $\mathbf{M}$ and its relation to $S$ shows that \cite{bk:higham-2008}
\[
\left.f(\mathbf{M})\right|_S = f\left( \left. \mathbf{M} \right|_S \right).
\]

\subsection{Numerical Range and Numerical Abscissa}
\label{sec:prelimes:nr}
The numerical range of $\mathbf{M}$ is the compact, convex subset of $\mathbb{C}$ given by \cite{bk:horn-1991}
\begin{equation}
W(\mathbf{M}) = \left\{ x \cdot{} \mathbf{M} x : x\in\mathbb{C}^N, \| x \| = 1 \right\}.
\label{eq:numerical-range}
\end{equation}
The numerical abscissa,
\begin{equation}
\omega(\mathbf{M}) = \max \left\{\mathrm{Re}(z) : z\in W(\mathbf{M}) \right\},
\label{eq:numerical-abscissa}
\end{equation}
comes from the boundary of $W(\mathbf{M})$ and relates to transient effects in $e^{\mathbf{M}t}$ \cite{bk:trefethen-2005}. An example of the numerical range and numerical abscissa is displayed in Figure \ref{fig:nr-ellipse}.

\begin{figure}
\centering
\resizebox{3.25in}{!}{\includegraphics{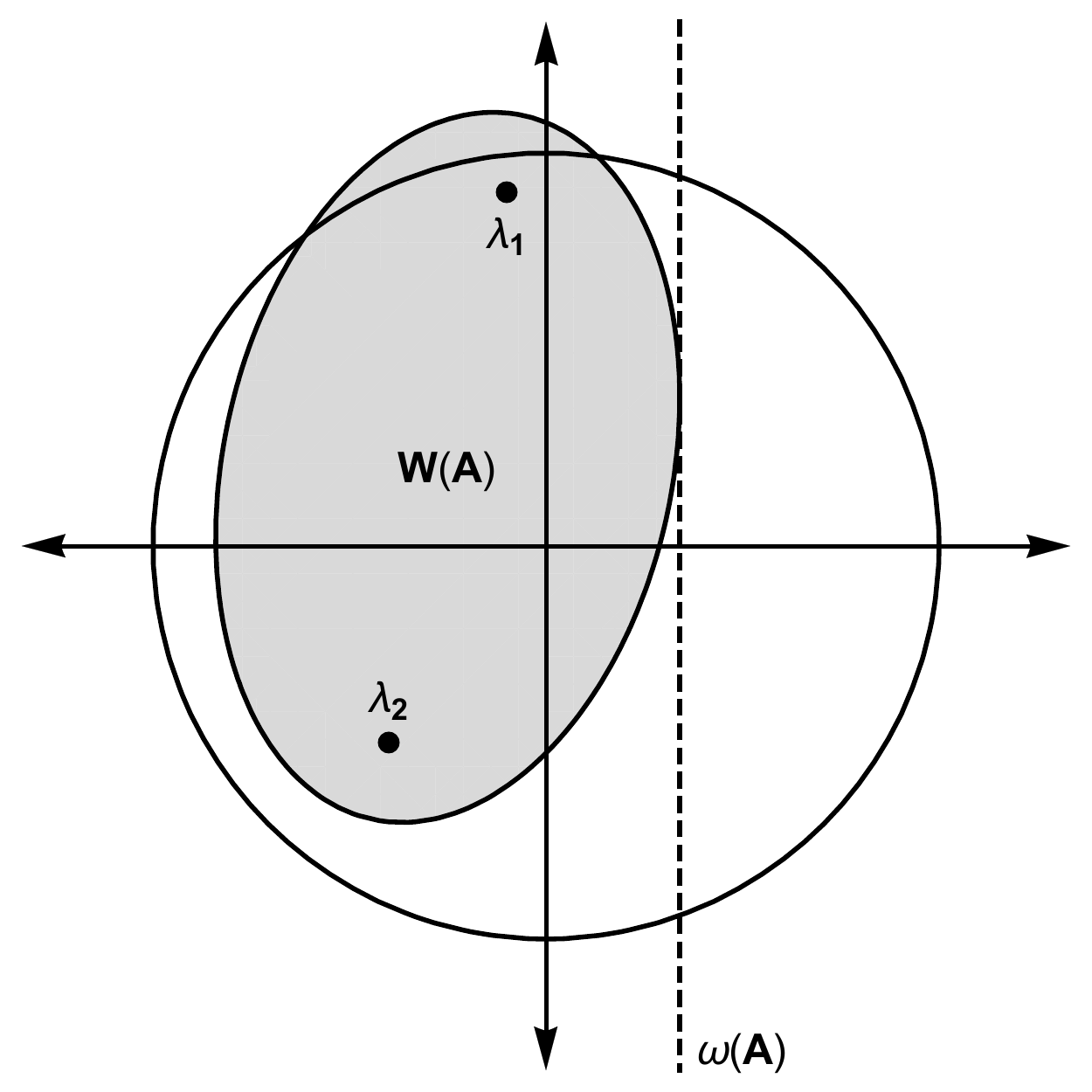}}
\caption{\label{fig:nr-ellipse}An example of the numerical range (shaded region) for $\mathbf{A}$ in Eq.\ \eqref{eq:matrix-a} with $\lambda_1=-0.1+0.9i$, $\lambda_2=-0.4-0.5i$, and $\theta=2\pi/7$. The real axis, imaginary axis, and unit circle are shown as solid lines, for reference. The dashed vertical line depicts $\mathrm{Re}(z)=\omega(\mathbf{A})$, the numerical abscissa of $\mathbf{A}$. Despite both eigenvalues being in the open left half of $\mathbb{C}$ (open unit circle), the numerical range extends into the right half plane (outside the unit circle), which signifies transient effects in the matrix exponential (matrix powers).}
\end{figure}

\subsection{Transient Effects in the Matrix Exponential and Matrix Powers}
\label{sec:prelims:transients}
Transient effects in the matrix exponential $e^{\mathbf{M}t}$ occur when all eigenvalues of $\mathbf{M}$ have negative real part and $\|e^{\mathbf{M}t}\| > 1$ for some $t>0$. Clearly, $\| e^{\mathbf{M}t} \|=1$ when $t=0$ such that a sufficient condition for transient behavior is
\[
\left. \frac{\mathrm{d}}{\mathrm{d}t} \left\| e^{\mathbf{M}t} \right\| \right|_{t=0^+} > 0.
\]
From Eq.\ (14.2) in \cite{bk:trefethen-2005},
\[
\left. \frac{\mathrm{d}}{\mathrm{d}t} \left\| e^{\mathbf{M}t} \right\| \right|_{t=0^+} = \omega(\mathbf{M}),
\]
Thus, a sufficient condition for the existence of transient effects in $e^{\mathbf{M}t}$ is
\begin{equation}
\omega(\mathbf{M}) > 0.
\label{eq:exist-transient-exp}
\end{equation}
Transient effects in the matrix exponential are consequently related to the boundary of $W(\mathbf{M})$.

Other bounds for predicting the existence of transient effects in $e^{\mathbf{M}t}$ can be found in chapter 14 of \cite{bk:trefethen-2005}, and involve (for example) the pseudospectra of $\mathbf{M}$ or the Kreiss constant of $\mathbf{M}$. These quantities are ``holistic'', meaning they examine $\mathbf{M}$ in its entirety. Our analyses in sections \ref{sec:2by2} and \ref{sec:higherD} will look at the relationship between specific invariant subspaces of $\mathbf{M}$ and transient effects, thereby providing a more fine-grained analysis of transient effects.

Transient effects in $\mathbf{M}^n$ are similar. In this case the eigenvalues of $\mathbf{M}$ are less than one in magnitude and $\| \mathbf{M}^n \|>1$ for some $n>0$. As before, a discussion of transient effects in the matrix powers can be found in chapter 14 of \cite{bk:trefethen-2005}, which includes bounds that depend on holistic properties of $\mathbf{M}$.

\section{Analytical Results on $2\times 2$ Matrices}
\label{sec:2by2}
Many of the conditions in Thms.\ \ref{thm:transient-exp} and \ref{thm:transient-power} come from analytical results on transient effects in $2\times2$ matrices. This section develops these results. First in section \ref{sec:2by2:nr} we discuss the numerical range of $2\times 2$ matrices \cite{uhlig-541-1985, bk:horn-1991}. Section \ref{sec:2by2:parameterize} then parameterizes $\mathbb{C}^{2\times 2}$ in a way that facilitates our analyses. The key parameter is the angle between the matrix's eigenvectors (assuming the matrix is not defective). Finally, sections \ref{sec:2by2:exp} and \ref{sec:2by2:power} derive sufficient conditions for the matrix exponential and matrix powers of a $2\times 2$ matrix, respectively, to display transient behavior.

\subsection{The Numerical Range}
\label{sec:2by2:nr}
The numerical range of a $2\times 2$ matrix $\mathbf{M}$ is a (possibly-degenerate) ellipse with boundary given by Thm.\ \ref{thm:nr-ellipse} \cite{uhlig-541-1985, bk:horn-1991}. The center of the ellipse is $\mathrm{tr}(\mathbf{M})/2$, the eigenvalues of $\mathbf{M}$ are the foci of the ellipse, and the major and minor axes are stated in Lemma \ref{lemma:axes-ellipse}. Many of the ensuing results will come from analyzing the geometry of this ellipse, for which an example is displayed in Figure \ref{fig:nr-ellipse}.

\begin{thm}[Boundary of the Numerical Range of a $2\times2$ Matrix, Taken from Thm.\ 2 in \cite{uhlig-541-1985}] 
In the complex plane $\mathbb{C}\simeq \mathbb{R}^2$, the boundary curve of the numerical range of $\mathbf{M}\in\mathbb{C}^{2\times2}$ is
\begin{equation}
\left[ x - \mathrm{Re}\left( \frac{\mathrm{tr}(\mathbf{M})}{2} \right), y - \mathrm{Im}\left( \frac{\mathrm{tr}(\mathbf{M})}{2} \right) \right] \mathbf{S} \left[ \begin{array}{c} x - \mathrm{Re}\left( \frac{\mathrm{tr}(\mathbf{M})}{2} \right) \\ y - \mathrm{Im}\left( \frac{\mathrm{tr}(\mathbf{M})}{2} \right) \end{array} \right] = \frac{\det(\mathbf{S})}{4},
\label{eq:nr-ellipse}
\end{equation}
where $\left[ \begin{array}{c} x \\ y \end{array} \right]\in\mathbb{R}^2$,
\[
\mathbf{S} = \left[ \begin{array}{cc}
\| \mathbf{M}_0 \|_\mathrm{F}^2 + 2 \, \mathrm{Re}(\det(\mathbf{M}_0)) & 2 \, \mathrm{Im}(\det(\mathbf{M}_0)) \\
2 \, \mathrm{Im}(\det(\mathbf{M}_0)) & \| \mathbf{M}_0 \|_\mathrm{F}^2 - 2 \, \mathrm{Re}(\det(\mathbf{M}_0))
\end{array} \right],
\]
$\mathbf{M}_0 = \mathbf{M}-(\mathrm{tr}(\mathbf{M})/2)\mathbf{I}$, and $\|\cdot\|_\mathrm{F}$ is the Frobenius norm.
\label{thm:nr-ellipse}
\end{thm}

\begin{lemma}[Major and Minor Axes of the Numerical Range of a $2\times 2$ Matrix, Taken from Thm.\ 1 in \cite{uhlig-541-1985}]
The major and minor axes of the elliptical numerical range of $\mathbf{M}\in\mathbb{C}^{2\times 2}$ are
\begin{subequations}
\begin{equation}
\label{eq:nr-major-axis}
2\sqrt{\left\| \mathbf{M}_0 \right\|_\mathrm{F}^2 + 2 \left| \mathrm{det}(\mathbf{M}_0) \right|}
\end{equation}
and
\begin{equation}
\label{eq:nr-minor-axis}
2\sqrt{\left\| \mathbf{M}_0 \right\|_\mathrm{F}^2 - 2 \left| \mathrm{det}(\mathbf{M}_0) \right|},
\end{equation}
\end{subequations}
respectively, where $\mathbf{M}_0 = \mathbf{M}-(\mathrm{tr}(\mathbf{M})/2)\mathbf{I}$, and $\|\cdot\|_\mathrm{F}$ is the Frobenius norm.
\label{lemma:axes-ellipse}
\end{lemma}

\subsection{Parameterizing $2\times 2$ Non-Normal Matrices}
\label{sec:2by2:parameterize}
Before developing the results for $2\times 2$ matrices, we first need to parameterize $\mathbb{C}^{2\times 2}$ using quantities that will facilitate our analysis. Appealing to the Jordan form, there are three cases.

First is when the matrix is diagonalizable and has a degenerate eigenvalue. In this case the matrix is a scalar multiple of the identity matrix; it is normal and not of interest to this discussion. We will not consider this case further.

Second is when the matrix is diagonalizable but with distinct eigenvalues, $\lambda_1$ and $\lambda_2$. We will denote this matrix by $\mathbf{A}$. From section \ref{sec:prelims:schur} we can assume without loss of generality that $\mathbf{A}$ is upper triangular:
\[
\mathbf{A} = \left[ \begin{array}{cc}
\lambda_1 & a \\
0 & \lambda_2
\end{array} \right],
\]
where $a,\lambda_1,\lambda_2\in\mathbb{C}$ and $\lambda_1\neq\lambda_2$. It is straightforward to see that $\left[ \begin{array}{cc} 1 & 0 \end{array} \right]^T$ is a right eigenvector associated with $\lambda_1$ and that $\left[ \begin{array}{cc} a & \lambda_2-\lambda_1 \end{array} \right]^T$ is a right eigenvector associated with $\lambda_2$. Then, the angle between the eigenvectors [Eq.\ \eqref{eq:vector-angle}] is
\begin{equation}
\cos(\theta) = \frac{|a|}{\sqrt{|a|^2+|\lambda_2-\lambda_1|^2}},
\label{eq:matrix-a:a-theta}
\end{equation}
which yields
\[
|a| = |\lambda_2-\lambda_1|\cot(\theta).
\]
Then, for $0\le\varphi<2\pi$, we get
\begin{equation}
\mathbf{A} = \left[ \begin{array}{cc}
\lambda_1 & e^{i\varphi} \left|\lambda_2-\lambda_1\right| \cot(\theta) \\
0 & \lambda_2
\end{array} \right].
\label{eq:matrix-a}
\end{equation}
The angle between the eigenvectors can thus be used to parameterize the non-normality of $\mathbf{A}$. This idea of examining the angle between eigenvectors comes from the condition number of an eigenvalue \cite{bk:heath-2002}, which essentially considers the angle between the eigenvalue's left and right eigenvectors.

Third is when the matrix is defective, with sole eigenvalue $\lambda$. This matrix will be called $\mathbf{D}$. The Schur form is the $2\times2$ Jordan block because the generalized eigenvector can always be chosen to be orthogonal to the eigenvector. Thus,
\begin{equation}
\mathbf{D} = \left[ \begin{array}{cc}
\lambda & 1 \\
0 & \lambda
\end{array} \right].
\label{eq:matrix-d}
\end{equation}

Some useful properties of $\mathbf{A}$ and $\mathbf{D}$ are stated in the following Lemmas.

\begin{lemma}
The major and minor axes of $W(\mathbf{A})$ are
\begin{subequations}
\begin{equation}
2 \left| \lambda_1 - \lambda_2 \right| \csc(\theta)
\label{eq:major-axis-a}
\end{equation}
and
\begin{equation}
2 \left| \lambda_1 - \lambda_2 \right| \cot(\theta),
\label{eq:minor-axis-a}
\end{equation}
\end{subequations}
respectively.
\label{lemma:axes-a}
\end{lemma}
\begin{proof}
The proof is trivial following from Eq.\ \eqref{eq:matrix-a} and Lemma \ref{lemma:axes-ellipse}.
\end{proof}

\begin{lemma}
The numerical abscissa of $\mathbf{A}$ is
\begin{equation}
\omega(\mathbf{A}) = \frac{\sqrt{\left| \lambda_1 - \lambda_2 \right|^2 \left( 1+2\cot^2(\theta) \right) + \mathrm{Re}\left[ (\lambda_1-\lambda_2)^2 \right]}}{2\sqrt{2}} + \frac{\mathrm{Re}(\lambda_1 + \lambda_2)}{2}.
\label{eq:num-abscissa-a}
\end{equation}
\label{lemma:num-abscissa-a}
\end{lemma}
\begin{proof}
$\mathbf{S}$ in Thm.\ \ref{thm:nr-ellipse} is
\[
\mathbf{S} = \left[ \begin{array}{cc}
\left| \lambda_1 - \lambda_2 \right|^2 \left( \frac{1}{2} + \cot^2(\theta) \right) - \frac{1}{2} \mathrm{Re}\left[ \left(\lambda_1-\lambda_2\right)^2 \right] & - \frac{1}{2} \mathrm{Im}\left[ \left(\lambda_1-\lambda_2\right)^2 \right] \\
- \frac{1}{2} \mathrm{Im}\left[ \left(\lambda_1-\lambda_2\right)^2 \right] & \left| \lambda_1 - \lambda_2 \right|^2 \left( \frac{1}{2} + \cot^2(\theta) \right) + \frac{1}{2} \mathrm{Re}\left[ \left(\lambda_1-\lambda_2\right)^2 \right]
\end{array} \right].
\]
Then, recast Eq.\ \eqref{eq:numerical-abscissa} as a constrained optimization problem: Maximize $x$ subject to Eq.\ \eqref{eq:nr-ellipse}. This can be solved using standard techniques, \textit{e.g.},\ Lagrange multipliers, which produces the result.
\end{proof}

\begin{lemma}
\label{lemma:num-abscissa-d}
The numerical abscissa of $\mathbf{D}$ is $\omega(\mathbf{D}) = \mathrm{Re}(\lambda) + 1/2$.
\end{lemma}
\begin{proof}
The proof is very similar to that of Lemma \ref{lemma:num-abscissa-a}. The matrix $\mathbf{S}$ from Thm.\ \ref{thm:nr-ellipse} is the identity matrix such that the boundary of $W(\mathbf{D})$ is
\[
\left(x-\mathrm{Re}(\lambda) \right)^2 + \left( y - \mathrm{Im}(\lambda) \right)^2 = \frac{1}{4}.
\]
The right-most point of the boundary curve is trivially $(\mathrm{Re}(\lambda)+1/2,\mathrm{Im}(\lambda))$; thus, $\omega(\mathbf{D}) = \mathrm{Re}(\lambda)+1/2$. Note that a circular numerical range is expected for nilpotent operators \cite{karaev-2321-2004} and $W(\mathbf{D})$ having a radius of $1/2$ is consistent with \cite{haagerup-371-1992, wu-351-1998}; see Thm.\ \ref{thm:nr-jordan} below.
\end{proof}

\subsection{Matrix Exponential ($2\times 2$ Matrices)}
\label{sec:2by2:exp}
As discussed in section \ref{sec:prelims:transients}, $e^{\mathbf{M}t}$ has transient behavior when $\omega(\mathbf{M}) > 0$. We can then examine the non-defective matrix $\mathbf{A}$ in Eq.\ \eqref{eq:matrix-a} and the defective matrix $\mathbf{D}$ in Eq.\ \eqref{eq:matrix-d} using this condition.

\begin{lemma}
\label{lemma:2by2:theta-exp}
Let $\mathbf{A}$ be given as in Eq.\ \eqref{eq:matrix-a} with $\lambda_1\neq\lambda_2$ and $\mathrm{Re}(\lambda_1),\mathrm{Re}(\lambda_2)<0$. Then, $e^{\mathbf{A}t}$ will have transient effects when
\[
\theta < \arctan\left( \frac{\left| \lambda_1 - \lambda_2 \right|}{2 \sqrt{\mathrm{Re}(\lambda_1) \mathrm{Re}(\lambda_2)}} \right).
\]
\end{lemma}
\begin{proof}
Use Lemma \ref{lemma:num-abscissa-a} to solve $\omega(\mathbf{A}) > 0$ for $\theta$.
\end{proof}

\begin{cor}
\label{cor:2by2:a-theta}
Let $\mathbf{M}\in\mathbb{C}^{2\times2}$ have the form (Schur form)
\[
\mathbf{M} = \left[ \begin{array}{cc}
\lambda_1 & a \\
0 & \lambda_2 \\
\end{array} \right],
\]
with $\lambda_1\neq\lambda_2$ and $\mathrm{Re}(\lambda_1),\mathrm{Re}(\lambda_2)<0$. Then, $e^{\mathbf{M}t}$ will have transient effects when $|
a|>2\sqrt{\mathrm{Re}(\lambda_1)\mathrm{Re}(\lambda_2)}$.
\end{cor}
\begin{proof}
The earlier discussion [Eq.\ \eqref{eq:matrix-a:a-theta}] relates $a$ to $\theta$, the angle between the eigenvectors of $\mathbf{M}$. Using trigonometry,
\[
\theta = \arctan\left( \frac{|\lambda_2-\lambda_1|}{|a|} \right),
\]
which provides the result when combined with Lemma \ref{lemma:2by2:theta-exp}.
\end{proof}

\begin{lemma}
\label{lemma:2by2:defective-exp}
Let $\mathbf{D}$ be given as in Eq.\ \eqref{eq:matrix-d} with $\mathrm{Re}(\lambda)<0$. Then, $e^{\mathbf{D}t}$ will have transient effects when $\mathrm{Re}(\lambda)>-1/2$.
\end{lemma}
\begin{proof}
Use Lemma \ref{lemma:num-abscissa-d} to solve $\omega(\mathbf{D}) > 0$ for $\mathrm{Re}(\lambda)$.
\end{proof}

\subsection{Matrix Powers ($2\times 2$ Matrices)}
\label{sec:2by2:power}
Let us now turn our attention to matrix powers $\mathbf{M}^n$. Because the existence of $\mathbf{M}^n$ is complicated for singular $\mathbf{M}$, we will restrict our attention to nonsingular $\mathbf{M}$. The results here follow straightforwardly from Lemmas \ref{lemma:2by2:theta-exp} and \ref{lemma:2by2:defective-exp} for the matrix exponential.

\begin{lemma}
\label{lemma:2by2:theta-power}
Let $\mathbf{A}$ be nonsingular and given as in Eq.\ \eqref{eq:matrix-a} with $\lambda_1\neq\lambda_2$ and $0<|\lambda_1|,|\lambda_2|<1$. Then, $\mathbf{A}^n$ will have transient effects when
\[
\theta < \arctan\left( \frac{\left| \ln (\lambda_1) - \ln(\lambda_2) \right|}{2 \sqrt{ \ln |\lambda_1| \ln |\lambda_2| } } \right).
\]
\end{lemma}
\begin{proof}
Because $\mathbf{A}$ is not singular, we can write $\mathbf{A}^n=e^{n\ln(\mathbf{A})}$ by choosing the logarithmic branch cut to avoid both eigenvalues \cite{bk:higham-2008}. The result then trivially follows from Lemma \ref{lemma:2by2:theta-exp} by using $\ln(\lambda_1)$ and $\ln(\lambda_2)$ as eigenvalues. Note that the logarithmic branch cut can always be chosen to make the imaginary part of $\ln(\lambda_1)-\ln(\lambda_2)$ be in $[-\pi,\pi)$.
\end{proof}

\begin{lemma}
\label{lemma:2by2:defective-power}
Let $\mathbf{D}$ be nonsingular and given as in Eq.\ \eqref{eq:matrix-d} with $0<|\lambda|<1$. Then, $\mathbf{D}^n$ will have transient effects.
\end{lemma}
\begin{proof}
The proof is similar to that of Lemmas \ref{lemma:num-abscissa-d} and \ref{lemma:2by2:theta-power}. Because $\lambda\neq0$, $\mathbf{D}^n=e^{n\ln(\mathbf{D})}$. It can be shown that \cite{bk:higham-2008}
\[
\ln(\mathbf{D}) = \left[ \begin{array}{cc} \ln(\lambda) & \lambda^{-1} \\ 0 & \ln(\lambda) \end{array} \right],
\]
which leads to $\mathbf{S} = \mathbf{I}/|\lambda|^2$ in Thm.\ \ref{thm:nr-ellipse}. Noting that $\mathrm{Re}(\ln(\lambda)) = \ln|\lambda|$, the boundary of $W(\ln(\mathbf{D}))$ is $\left(x-\ln|\lambda|\right)^2 + \left( y - \mathrm{arg}(\lambda) \right)^2 = 1/(4|\lambda|^2)$. Thus, 
\[
\omega(\ln(\mathbf{D}))=\ln|\lambda| + \frac{1}{2|\lambda|}.
\]
One can verify that $\omega(\ln(\mathbf{D}))>0$ for $0<|\lambda|<1$, implying that there will always be transient effects.
\end{proof}

Lemmas \ref{lemma:2by2:theta-exp}, \ref{lemma:2by2:defective-exp}, \ref{lemma:2by2:theta-power}, and \ref{lemma:2by2:defective-power} collectively summarize sufficient conditions for $\mathbf{M}\in\mathbb{C}^{2\times 2}$ to have transient effects in its matrix exponential or matrix powers ($\mathbf{M}$ nonsingular assumed for the matrix powers). Note that Lemmas \ref{lemma:2by2:theta-power} and \ref{lemma:2by2:defective-power} consider $n$ to be continuous such that transient effects may not be noticeable when $n$ is restricted to positive integers.

\begin{remark}
Although not proven, numerical experiments suggest that the conditions in Lemmas \ref{lemma:2by2:theta-exp}, \ref{lemma:2by2:defective-exp}, and \ref{lemma:2by2:theta-power} are also necessary for transient effects in $e^{\mathbf{M}t}$ or $\mathbf{M}^n$ when $\mathbf{M}\in\mathbb{C}^{2\times 2}$. Future work should be performed to confirm or refute this conjecture.
\end{remark}

\section{Higher-Dimensional Matrices (Proofs of Thms.\ \ref{thm:transient-exp} and \ref{thm:transient-power})}
\label{sec:higherD}
The conditions for transient effects in $2\times 2$ matrices that we developed in the previous section readily generalize to $N\times N$ matrices by considering the 2-dimensional invariant subspaces of $\mathbf{M}\in\mathbb{C}^{N\times N}$. Theorem \ref{thm:subspaces} states the relationship.

\begin{thm}
Let $\mathbf{M}\in\mathbb{C}^{N \times N}$ and let $S\subseteq\mathbb{C}^N$ be an invariant subspace of $\mathbf{M}$. If $e^{\left.\mathbf{M}\right|_S t}$ admits transient behavior, then so will $e^{\mathbf{M}t}$. A similar statement holds for $\left( \left. \mathbf{M} \right|_S \right)^n$ and $\mathbf{M}^n$.
\label{thm:subspaces}
\end{thm}
\begin{proof}
Transient effects in $e^{\mathbf{M}t}$ require $\|e^{\mathbf{M}t}\|>1$ for some $t>0$. Because $\left.\mathbf{M}\right|_S$ admits transient behavior, we know that $1<\| e^{\left.\mathbf{M}\right|_St} \|$ for some $t>0$. Then, using the definition of the matrix norm,
\[
1 < \left\| e^{\left. \mathbf{M} \right|_S t} \right\| = \max_{\|x\|=1} \left\| e^{\left. \mathbf{M} \right|_S t}x \right\| = \max_{x\in S, \|x\|=1} \left\| e^{\mathbf{M}t}x \right\| \le \max_{\|x\|=1} \left\| e^{\mathbf{M}t}x \right\| = \left\| e^{\mathbf{M}t} \right\|.
\]
Thus, if $\left.\mathbf{M}\right|_S$ admits transient behavior in the matrix exponential, so must $\mathbf{M}$. Similar logic is used to show that if $\left( \left. \mathbf{M} \right|_S  \right)^n$ has transient effects, then $\mathbf{M}^n$ will as well.
\end{proof}

Using Thm.\ \ref{thm:subspaces}, Lemma \ref{lemma:2by2:theta-exp} proves the second condition of Thm.\ \ref{thm:transient-exp} and Lemmas \ref{lemma:2by2:defective-power} and \ref{lemma:2by2:theta-power} prove the first and second conditions of Thm.\ \ref{thm:transient-power}, respectively. Lemma \ref{lemma:2by2:defective-exp} shows the first condition of Thm.\ \ref{thm:transient-exp} for 2-dimensional Jordan subspaces. \cite{haagerup-371-1992, wu-351-1998} allow us to consider Jordan subspaces of arbitrary dimension because the numerical range is analytically known for these cases.

\begin{thm}[Numerical Range of a Jordan Block \cite{haagerup-371-1992, wu-351-1998}]
Let $\mathbf{J}$ be a $M\times M$ Jordan block ($M\ge 2$) with eigenvalue $\lambda$. Then $W(\mathbf{J})$ is all points $z\in\mathbb{C}$ satisfying
\[
\left| z - \lambda \right| \le \cos\left( \frac{\pi}{M+1} \right).
\]
Consequently,
\[
\omega(\mathbf{J})= \mathrm{Re}(\lambda) + \cos\left( \frac{\pi}{M+1} \right).
\]
\label{thm:nr-jordan}
\end{thm}

The first condition in Thm.\ \ref{thm:transient-exp} trivially follows from Eq.\ \eqref{eq:exist-transient-exp} and Thm.\ \ref{thm:nr-jordan}. This completes the proofs of Thms.\ \ref{thm:transient-exp} and \ref{thm:transient-power}.

\section{Geometric Interpretations}
\label{sec:geom}
The results in section \ref{sec:2by2} largely came from analyzing a matrix $\mathbf{M}\in\mathbb{C}^{2\times 2}$ in terms of the angle between its eigenvectors (assuming $\mathbf{M}$ is not defective). This alone leads to new geometric interpretations for transient effects in the matrix exponential and matrix powers of $\mathbf{M}$, but two other geometric observations can be made as well.

First regards the numerical detection of a defective eigenvalue. Due to numerical instabilities in computing the eigenvalues of a non-normal matrix \cite{bk:trefethen-2005} --- especially a defective eigenvalue --- it can be very difficult to distinguish (\textit{e.g.})\ two eigenvalues that are close to each other from a defective eigenvalue. A recent discussion of this issue appears in \cite{zeng-798-2016}. Let $S$ be the 2-dimensional invariant subspace of $\mathbf{M}$ corresponding to these two eigenvalues. An examination of $W\left(\left. \mathbf{M}\right|_S \right)$ provides insight into the situation. From Lemma \ref{lemma:axes-a}, the major and minor axes of $W\left( \left. \mathbf{M} \right|_S \right)$ are $2|\lambda_1-\lambda_2|\csc(\theta)$ and $2|\lambda_1-\lambda_2|\cot(\theta)$, respectively. As $\theta\to0^+$, meaning $\left.\mathbf{M}\right|_S$ approaches a defective matrix, $W\left( \left. \mathbf{M} \right|_S \right)$ will tend toward a circle. Examining the shape of the restricted numerical range for the invariant subspace may help determine if it is a Jordan subspace or not.

Second is a relation to antieigenvalue analyses \cite{bk:gustafson-2012}. When $\mathbf{M}$ is a self-adjoint, positive-definite matrix (and thus normal), its antieigenvector is the vector that is rotated the most by $\mathbf{M}$. The antieigenvalue is the corresponding angle of rotation, which will be in $[0,\pi/2)$. If $\mathbf{M}$ is $2\times 2$, self-adjoint, and negative-definite with eigenvalues $\lambda_1$ and $\lambda_2$, its antieigenvalue is \cite{bk:gustafson-2012}
\[
\arcsin\left( \frac{|\lambda_1 - \lambda_2|}{|\lambda_1 + \lambda_2|} \right),
\]
where the angle of rotation is measured in the Hermitian sense of Eq.\ \eqref{eq:vector-angle}. Compare this against the critical angle where $\omega(\mathbf{A})=0$ in Lemma \ref{lemma:2by2:theta-exp}, which is
\[
\arctan\left( \frac{\left| \lambda_1 - \lambda_2 \right|}{2\sqrt{\lambda_1\lambda_2}} \right)
\]
when $\lambda_1,\lambda_2\in\mathbb{R}$ and $\lambda_1,\lambda_2<0$. Some basic trigonometry reveals that these two angles are the same; the critical angle for transient effects in $e^{\mathbf{M}t}$ is exactly the antieigenvalue of a normal matrix with the same real eigenvalues. That is, transient effects appear in $e^{\mathbf{M}t}$ when the normal matrix with the same eigenvalues can rotate vectors more than the angle between the eigenvectors of $\mathbf{M}$. Future work may consider generalizing this connection to matrices with complex eigenvalues, for which an antieigenvalue analysis is not well developed. 

\section{Concluding Remarks}
\label{sec:conclusions}
In this work we have established a relationship (Thm.\ \ref{thm:subspaces}) between the invariant subspaces of a non-normal matrix $\mathbf{M}\in\mathbb{C}^{N\times N}$ and transient behavior in $e^{\mathbf{M}t}$ ($t>0$) and $\mathbf{M}^n$ ($n>0$). This then allowed us to derive sufficient conditions for transient effects in the exponential and powers of $\mathbf{M}$, as stated in Thms.\ \ref{thm:transient-exp} and \ref{thm:transient-power}, respectively. Most of the conditions came from analytic results for 2-dimensional invariant subspaces of $\mathbf{M}$, where the angle between the two associated eigenvectors is a key parameter. We should note that the conditions in Thms.\ \ref{thm:transient-exp} and \ref{thm:transient-power} are sufficient but not necessary for transient effects; it is possible that transient effects can only be associated with higher-dimensional invariant subspaces.

Consequently, one immediate next step is to analyze higher-dimensional invariant subspaces for conditions on transient effects. Although not as straightforward as the case for $2\times 2$ matrices, the numerical range of $3\times 3$ matrices is discussed in \cite{keeler-115-1997}. Their results can be used with Eq.\ \eqref{eq:exist-transient-exp} to develop appropriate conditions. Along these lines, the Jordan principal angles \cite{galantai-589-2006} may be useful when generalizing the angle between eigenvectors to angles between invariant subspaces of arbitrary dimensions.

In a broader sense such investigations are examining structure within $W(\mathbf{M})$ \cite{uhlig-055019-2008, dunkl-2042-2011}. For instance, what is the lowest-dimensional invariant subspace $S$ of $\mathbf{M}$ that satisfies $\omega\left(\left.\mathbf{M}\right|_S\right)>0$, thereby inducing transient behavior? Ideas on the restricted numerical range \cite{das-35-1987, gawron-102204-2010} may help answer this question. Finally, we end with a practical comment. All of our results make use of the Jordan decomposition of a non-normal $\mathbf{M}$, which can be difficult to compute accurately \cite{bk:trefethen-2005}. Although the results are theoretically useful, additional work needs to be performed to strengthen their practical utility, possibly by further exploring connections to the Schur decomposition as in Cor.\ \ref{cor:2by2:a-theta}.

\section*{Acknowledgments}
I thank Jay Bardhan, Thorsten Hansen, Dominik Orlowski, Christopher DeGrendele, and Jonathan Kazakov for helpful conversations. 

\bibliographystyle{elsarticle-num}
\bibliography{refs}

\begin{thebibliography}{10}
\expandafter\ifx\csname url\endcsname\relax
  \def\url#1{\texttt{#1}}\fi
\expandafter\ifx\csname urlprefix\endcsname\relax\def\urlprefix{URL }\fi
\expandafter\ifx\csname href\endcsname\relax
  \def\href#1#2{#2} \def\path#1{#1}\fi

\bibitem{moler-3-2003}
C.~Moler, C.~Van~Loan, {Nineteen Dubious Ways to Compute the Exponential of a
  Matrix, Twenty-Five Years Later}, SIAM Rev. 45~(1) (2003) 3--49.

\bibitem{bk:trefethen-2005}
L.~N. Trefethen, M.~Embree, {Spectra and Pseudospectra}, The Behavior of
  Nonnormal Matrices and Operators, Princeton University Press, Princeton, NJ,
  USA, 2005.

\bibitem{reuter-053001-2017}
M.~G. Reuter, {A Unified Perspective of Complex Band Structure:
  Interpretations, Formulations, and Applications}, J. Phys.: Condens. Matter
  29 (2017) 053001.

\bibitem{bk:gohberg-2006}
I.~Gohberg, P.~Lancaster, L.~Rodman, {Invariant Subspaces of Matrices with
  Applications}, Society for Industrial and Applied Mathematics, Philadelphia,
  PA, USA, 2006.

\bibitem{galantai-589-2006}
A.~Gal{\'a}ntai, C.~J. Heged{\H{u}}s, {Jordan's Principal Angles in Complex
  Vector Spaces}, Numer. Lin. Alg. Appl. 13~(7) (2006) 589--598.

\bibitem{bk:golub-2013}
G.~H. Golub, C.~F. Van~Loan, {Matrix Computations}, 4th Edition, The Johns
  Hopkins University Press, Baltimore, MD, USA, 2013.

\bibitem{bk:higham-2008}
N.~J. Higham, {Functions of Matrices}, Theory and Computation, SIAM,
  Philadelphia, PA, 2008.

\bibitem{bk:horn-1991}
R.~A. Horn, C.~R. Johnson, {Topics in Matrix Analysis}, Cambridge University
  Press, Cambridge, UK, 1991.

\bibitem{uhlig-541-1985}
F.~Uhlig, {The Field of Values of a Complex Matrix, An Explicit Description in
  the 2{\texttimes}2 Case}, SIAM J. Alg. Disc. Math. 6~(4) (1985) 541--545.

\bibitem{bk:heath-2002}
M.~T. Heath, {Scientific Computing: An Introductory Survey}, 2nd Edition,
  McGraw-Hill, New York, NY, USA, 2002.

\bibitem{karaev-2321-2004}
M.~T. Karaev, {The Numerical Range of a Nilpotent Operator on a Hilbert Space},
  Proc. Am. Math. Soc. 132~(8) (2004) 2321--2326.

\bibitem{haagerup-371-1992}
U.~Haagerup, P.~de~la Harpe, {The Numerical Radius of a Nilpotent Operator on a
  Hilbert Space}, Proc. Am. Math. Soc. 115~(2) (1992) 371--379.

\bibitem{wu-351-1998}
P.~Y. Wu, {A Numerical Range Characterization of Jordan Blocks}, Lin. Multilin.
  Alg. 43 (1998) 351--361.

\bibitem{zeng-798-2016}
Z.~Zeng, {Sensitivity and Computation of a Defective Eigenvalue}, SIAM J.
  Matrix Anal. Appl. 37~(2) (2016) 798--817.

\bibitem{bk:gustafson-2012}
K.~Gustafson, {Antieigenvalue Analysis}, World Scientific, Hackensack, NJ, USA,
  2012.

\bibitem{keeler-115-1997}
D.~S. Keeler, L.~Rodman, I.~M. Spitkovsky, {The Numerical Range of $3\times 3$
  Matrices}, Lin. Alg. Appl. 252 (1997) 115--139.

\bibitem{uhlig-055019-2008}
F.~Uhlig, {An Inverse Field of Values Problem}, Inv. Prob. 24~(5) (2008)
  055019.

\bibitem{dunkl-2042-2011}
C.~F. Dunkl, P.~Gawron, J.~A. Holbrook, Z.~Pucha{\l}a, K.~{\.{Z}}yczkowski,
  {Numerical Shadows: Measures and Densities on the Numerical Range}, Lin. Alg.
  Appl. 434~(9) (2011) 2042--2080.

\bibitem{das-35-1987}
K.~C. Das, S.~Majumdar, B.~Sims, {Restricted Numerical Range and Weak
  Convergence}, J. Math. Phys. Sci. 21~(1) (1987) 35--42.

\bibitem{gawron-102204-2010}
P.~Gawron, Z.~Pucha{\l}a, J.~A. Miszczak, {\L}.~Skowronek, K.~{\.{Z}}yczkowski,
  {Restricted Numerical Range: A Versatile Tool in the Theory of Quantum
  Information}, J. Math. Phys. 51~(10) (2010) 102204.

\end{thebibliography}

\end{document}